\documentclass[10pt]{article}
\usepackage{graphicx} 
\usepackage{graphics}
\usepackage{epsfig} 
\usepackage{amsmath} 
\usepackage{amssymb}  
\usepackage{mathtools}
\usepackage{xcolor}
\usepackage{colortbl}
\usepackage{dblfloatfix}    
\usepackage{arydshln}
\usepackage{bm}
\usepackage{relsize}
\usepackage{tikz}
\usetikzlibrary{backgrounds,matrix,arrows.meta,decorations.pathreplacing,calligraphy}
\usepackage{circuitikz}
\usepackage{etoolbox}
\usepackage{ctable}
\usepackage{amsthm}
\usepackage[utf8]{inputenc}
\usepackage[english]{babel}
\usepackage{footmisc}
\definecolor{Gray}{gray}{0.85}
\newcolumntype{a}{>{\columncolor{Gray}}c}
\title{On the Inverse of Forward Adjacency Matrix}
\date{\vspace{-5ex}}
\author{Pritam Mukherjee\\National University of Singapore\\mukherjee.hve@gmail.com \and L. Satish\\Indian Institute of Science, Bangalore\\satish@iisc.ac.in}
\DeclareMathSizes{11}{19}{13}{9}
\newtheorem{theorem}{Theorem}
\newtheorem{corollary}{Corollary}[theorem]
\newtheorem{lemma}[theorem]{Lemma}
\theoremstyle{definition}
\newtheorem{definition}{Definition}[section]
\theoremstyle{remark}
\newtheorem*{remark}{Remark}

\newtheorem*{example}{Example}
\begin{document}
\maketitle
\thispagestyle{empty}
\pagestyle{empty}
\begin{abstract}
During routine state space circuit analysis of an arbitrarily connected set of nodes representing a lossless LC network, a matrix was formed that was observed to implicitly capture connectivity of the nodes in a graph similar to the conventional incidence matrix, but in a slightly different manner. This matrix has only 0, 1 or -1 as its elements. A sense of direction (of the graph formed by the nodes) is inherently encoded in the matrix because of the presence of -1. It differs from the incidence matrix because of leaving out the datum node from the matrix. Calling this matrix as forward adjacency matrix, it was found that its inverse also displays useful and interesting physical properties when a specific style of node-indexing is adopted for the nodes in the graph. The graph considered is connected but does not have any closed loop/cycle (corresponding to closed loop of inductors in a circuit) as with its presence the matrix is not invertible. Incidentally, by definition the graph being considered is a tree. The properties of the forward adjacency matrix and its inverse, along with rigorous proof, are presented.
\end{abstract}

\section{Introduction}\label{sec:introduction}
Incidence matrix or nodal incidence matrix is frequently encountered in graph theory \cite{gross_yellen}. This matrix essentially captures how nodes are interconnected in an undirected graph. However, while gathering the differential equations of a circuit formed by Kirchoff's voltage and current laws (or analogous equations in lumped mechanical circuits) which is also known as state space modelling \cite{chen,khalil,hinrichsen_pritchard}, often a slightly different matrix is encountered which has "0", "1" or "-1" as its entries. After defining the graph, the matrix will be formally defined in subsequent sections. Before that let us call it as \emph{Forward adjacency matrix} and denote it by $\mathbf{A}$. It is observed that $\mathbf{A}$ entirely upper-triangular. The most noteworthy features of $\mathbf{A}$ is the presence of "-1" which clearly indicates that it somehow preserves the sense of direction present in the graph within itself. This article proves an important feature of $\mathbf{A}^{-1}$.

As will be evident from the definition, $\mathbf{A}$ becomes singular if loops/cycles are present in the graph; therefore, the graph considered will be connected but acyclic. The considered graph, therefore, is essentially a tree by definition. Following Let a specific node be denoted as "$n$" with its index-number, say "$i$", as the subscript. Let the indexing strictly follow the rule:  \emph{if one travels from node-$n_0$ to any node-$n_i$, successively encountered node numbers must always increase}. For convenience, we start from a terminal node of a radial branch and work our way into the graph. This starting node is denoted as $n_0$. Note: the properties of $\mathbf{A}^{-1}$ , derived in this article, lose their significance if, in case, the nodes are randomly/arbitrarily numbered.

Note: The present work was independently carried out by the authors as a part of proving a fundamental theorem of circuit theory. It is regrettable that the authors were unaware of the previous efforts in literature that dealt with the same matrix and arrived at the same result, although following a different approach. As it happens, such a matrix has previously been studied and reported in the literature \cite{mayeda,miller,hale,resh,branin} and the same result is proved in this paper as well. However, a careful examination would reveal that the proof presented in this article is original and, unlike the previous works, established starting from the first principles of matrix inversions and linear algebra and so relatively much easier to follow. Therefore, the authors intend to retain this article on an open platform to present an alternative proof of an established result in graph theory. A critical comparison of the previous efforts with the present work is given below.

While works on similar matrices can be found in \cite{mayeda} and \cite{miller}, the pertinent property of the inverse of "a non-singular submatrix of the incidence matrix" is, to the best of authors' knowledge, due to Hale, Resh and Branin. Hale, in \cite{hale}, explains how to construct the inverse of the mentioned matrix avoiding brute-force matrix inversion but omits its "lengthy" proof. Resh \cite{resh} provides an interesting proof by adding edges to the graph followed by an elegant use of the orthogonality of incidence matrix and, in his nomenclature, "circuit matrices". Branin mentions in \cite{branin,branin_1} that such a result is a readily obtainable corollary of a previously proved theorem of his: \emph{The inverse of the branch-node matrix $A_T$ for any tree is equal to the transpose of the corresponding node-to-datum path matrix $B_T$}. He also mentions that the theorem was similarly but independently proved by Ponstein \cite{ponstein}. The proof of this theorem, as presented by Branin and Ponstein, and quite naturally, makes extensive use of the properties of topological matrices in graph theory. On the contrary, the present work assumes no prior knowledge of graph theory. Unfortunately, authors of this article were unaware of the previous efforts while submitting the initial version and, as it happens, this article proves the same result that Hale mentioned in \cite{hale}.

\section{The Graph}\label{sec:the_graph}
Consider the graph as shown in Fig.~\ref{fig:1} which shows nodes and edges. The direction shown on every edge represents moving towards increasing node-number. Apart from numbering nodes, important features of this graph are:
\begin{itemize}
	\item Edges do not construct a closed loop.
	\item There is no isolated node, i.e., for every node there exists at least one edge that is connected to it.
\end{itemize}

\begin{figure}[!h]
\centering
\begin{tikzpicture}
\draw (0,0)node[ocirc]{} to (0,-1)node[ocirc]{} to (0,-2) to (0,-3) to (0,-4)node[ocirc]{} to (0,-5)node[ocirc]{};
\draw (0,-2) to (1,-2)node[ocirc]{};
\draw (0,-2)node[ocirc]{} to (-1,-2)node[ocirc]{} to (-2,-2)node[ocirc]{};
\draw (0,-3)node[ocirc]{} to (-1,-3) to (-2,-3)node[ocirc]{};
\draw (-1,-3)node[ocirc]{} to (-1,-4)node[ocirc]{} to (-1,-5)node[ocirc]{};
\node [right=1pt] at (0,0) {$n_0$};
\node [right=1pt] at (0,-1) {$n_1$};
\node [right=1pt] at (0,-2.2) {$n_2$};
\node [right=1pt] at (0,-3) {$n_3$};
\node [right=1pt] at (0,-4) {$n_4$};
\node [right=1pt] at (0,-5) {$n_5$};
\node [right=1pt] at (1,-2) {$n_6$};
\node [above=1pt] at (-1,-2) {$n_7$};
\node [left=1pt] at (-2,-2) {$n_8$};
\node [above=1pt] at (-1,-3) {$n_9$};
\node [left=1pt] at (-2,-3) {$n_{10}$};
\node [left=1pt] at (-1,-4) {$n_{11}$};
\node [left=1pt] at (-1,-5) {$n_{12}$};
\filldraw[draw=black,fill=black] (0,-0.6) --+ (0.1,0.2) --+ (-0.1,0.2);
\filldraw[draw=black,fill=black] (0,-1.6) --+ (0.1,0.2) --+ (-0.1,0.2);
\filldraw[draw=black,fill=black] (0,-2.6) --+ (0.1,0.2) --+ (-0.1,0.2);
\filldraw[draw=black,fill=black] (0,-3.6) --+ (0.1,0.2) --+ (-0.1,0.2);
\filldraw[draw=black,fill=black] (0,-4.6) --+ (0.1,0.2) --+ (-0.1,0.2);
\filldraw[draw=black,fill=black] (0.75,-2) --+ (-0.2,0.1) --+ (-0.2,-0.1);
\filldraw[draw=black,fill=black] (-0.6,-2) --+ (0.2,0.1) --+ (0.2,-0.1);
\filldraw[draw=black,fill=black] (-1.6,-2) --+ (0.2,0.1) --+ (0.2,-0.1);
\filldraw[draw=black,fill=black] (-0.6,-3) --+ (0.2,0.1) --+ (0.2,-0.1);
\filldraw[draw=black,fill=black] (-1.6,-3) --+ (0.2,0.1) --+ (0.2,-0.1);
\filldraw[draw=black,fill=black] (-1,-3.6) --+ (0.1,0.2) --+ (-0.1,0.2);
\filldraw[draw=black,fill=black] (-1,-4.6) --+ (0.1,0.2) --+ (-0.1,0.2);
\end{tikzpicture}
\caption{The graph}
\label{fig:1}
\end{figure}
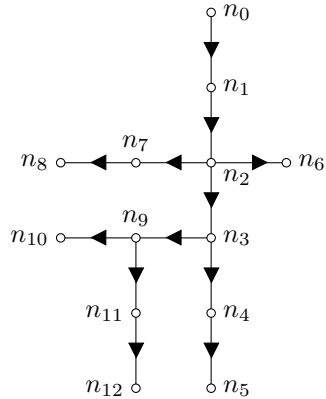

Parameters that are going to be extensively used hereafter are defined and discussed next.
\begin{definition}\label{definition:2_1}
Traveling in the direction of increasing node-number is denoted as \emph{traveling forward}.
\end{definition}
\begin{remark}
	Traveling forward does not necessarily have to begin from $n_0$. Further, it is not possible to travel forward from the last node of a radial branch, e.g, $n_5$, $n_6$, $n_8$, $n_{10}$ and $n_{12}$ in Fig.~\ref{fig:1}.
\end{remark}

\begin{definition}\label{definition:2_2}
	$n_i$ is said to be \emph{forward-adjacent} to $n_k$, if $n_i$ is adjacent to $n_k$ and $i<k$.
\end{definition}

\begin{definition}\label{definition:2_3}
	$n_i$ is said to be \emph{forward-connected} to $n_j$ if it is possible to go from $n_i$ to $n_j$ by strictly \emph{traveling forward}. The route for traveling forward from $n_i$ to $n_j$ is denoted as $R_{i\rightarrow j}$.
	\begin{remark}
		The node $n_i$ is forward-adjacent to $n_j$ is the simplest case of $n_i$ being forward connected to $n_j$. It is obvious that if $n_i$ is forward-connected but not forward-adjacent to $n_j$, there should exist one or more forward-adjacent nodes $n_{k_\beta}$ where $i<k_\beta<j$. Mathematically, in such a case, there will exist nodes $n_{k_1}$, $n_{k_2}$, ..., $n_{k_p}$ ($1\leq p<j-i$) \footnote[1]{The inequality $1\leq p<j-i$ comes because $n_i$ and $n_j$ cannot have more than $j-i-1$ nodes in between them. Consider $n_2$ and $n_5$ for example. The intermediate nodes are $n_3$ and $n_4$, so $5-2-1=2$. Of course, this number can be less than this. Consider $n_{10}$ and $n_2$ for instance. There are two intermediate nodes: $n_3$ and $n_9$ which is less than $10-2-1=7$.} such that $n_{k_\beta}$ is forward-adjacent to $n_{k_{\beta+1}} ~\forall~1\leq\beta\leq p-1$. According to Definition~\ref{definition:2_2}, this implies $i<k_1<k_2<...<k_p<j$.
	\end{remark}
\end{definition}

\begin{theorem}\label{theorem:1}
	For each $n_i$, there exists a unique $R_{0\rightarrow i}$. \footnote[2]{In other words, one route to travel-forward from $n_0$ to an arbitrary $n_i$ always exists and more than one route does not exist. This is important to prove as definition of the graph  does not readily imply this.}
	\begin{proof}
        Since there is no isolated node in the graph, existence of $R_{0\rightarrow i}$ is trivial. Therefore we focus on its uniqueness. Consider two possible routes $R_{0\rightarrow i}$ and $R_{0\rightarrow i}'$. Now, $R_{0\rightarrow i}$ and $R_{0\rightarrow i}'$ have common nodes $n_0$ and $n_i$. If there is no other common node as shown in Fig.~\ref{fig:2}(a) or if there are more common nodes as shown in Fig.~\ref{fig:2}(b), at least one closed-loop exists which violates the initial consideration. Therefore, $R_{0\rightarrow i}$ is unique.
        \begin{figure}[h]
          \centering
            \begin{tikzpicture}
              \draw (-6,0)node[left=1pt]{$n_0$} -- (-2,0)node[right=1pt]{$n_i$};
              \draw[dashed,blue] (-6,0) to (-5,-0.5) to (-3,-0.5) to (-2,0);
              \node [below=13pt] at (-4,0) {(a)};
              \draw (2,0)node[left=1pt]{$n_0$} -- (6,0)node[right=1pt]{$n_i$};
              \draw[dashed,blue] (2,0) to (2.8,-0.5) to (4.3,0.5) to (4.8,0) to (5.2,0) to (5.5,-0.5) to (6,0);
              \node [below=13pt] at (4,0) {(b)};
            \end{tikzpicture}
          \caption{Two hypothetical distinct routes to travel-forward from $n_0$ to $n_i$: (a) with only terminal nodes common, and (b) with more than two common nodes.}\label{fig:2}
        \end{figure}
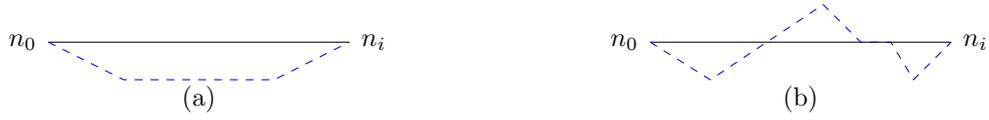
	\end{proof}
\end{theorem}

\begin{corollary}\label{corollary:1_1}
	$R_{i\rightarrow j}$ is unique if it exists. \footnote[3]{$R_{0\rightarrow i}$ is unique does not mean that $R_{i\rightarrow j}$ is unique. Therefore it needs to be proved.}
	\begin{proof}
		Let $R_{i\rightarrow j}$ and $R_{i\rightarrow j}'$ be two distinct routes to travel-forward from $n_i$ to $n_j$. Therefore, there would exist two distinct routes to travel-forward from $n_0$ to $n_j$: (i) a combination of $R_{0\rightarrow i}$ and $R_{i\rightarrow j}$, (ii) a combination of $R_{0\rightarrow i}$ and $R_{i\rightarrow j}'$ which is clearly a contradiction to Theorem~\ref{theorem:1}. Hence, $R_{i\rightarrow j}$ is unique.
	\end{proof}
\end{corollary}
\begin{remark}
	It is important to note that $R_{i\rightarrow j}$ may not exist for an arbitrary pair of $n_i$ and $n_j$. For instance, $R_{i\rightarrow j}$ cannot exist if $j\leq i$. Further, it can be seen in Figure~\ref{fig:1} that $R_{4\rightarrow 9}$ does not exist. In this case, one has to travel $n_4\rightarrow n_3\rightarrow n_9$ which is not traveling forward. Further, if $n_i$ is forward adjacent to $n_j$, $R_{i\rightarrow j}$ trivially exists.
\end{remark}

\begin{corollary}\label{corollary:1_2}
	More than one node cannot be forward-adjacent to a particular node.
	\begin{proof}
		Let there be two nodes $n_{i_1}$ and $n_{i_2}$ such that both are forward-adjacent to $n_j$. This implies that both $R_{i_1\rightarrow j}$ and $R_{i_2\rightarrow j}$ exist. Now there are two distinct routes to travel forward from $n_0$ to $n_j$: (i) combination of $R_{0\rightarrow i_1}$ and $R_{i_1\rightarrow j}$ and (ii) combination of $R_{0\rightarrow i_2}$ and $R_{i_2\rightarrow j}$, thus making $R_{0\rightarrow j}$ non-unique which is contradiction to Theorem~\ref{theorem:1}.
	\end{proof}
\end{corollary}

\begin{definition}\label{definition:2_4}
	If node $n_k$ is encountered in route $R_{i\rightarrow j}$ (inclusive of the terminal nodes $n_i$ and $n_j$), we denote $n_k\in R_{i\rightarrow j}$.
\end{definition}

\begin{corollary}\label{corollary:1_3}
	$n_i$ is forward-connected to $n_j$ if and only if $n_i\in R_{0\rightarrow j}$.
\end{corollary}
\begin{proof}
	\emph{Necessity}:
	According to Theorem~\ref{theorem:1}, $R_{0\rightarrow i}$ exists. If $n_i$ is forward connected to $n_j$, $R_{i\rightarrow j}$ exists. $R_{0\rightarrow i}$ and $R_{i\rightarrow j}$ together construct a route to travel forward from  $n_0$ to $n_j$. According to Theorem~\ref{theorem:1}, this is the only route, i.e., $R_{0\rightarrow j}$ and clearly it contains node-$n_i$. Thus, $n_i\in R_{0\rightarrow j}$ is necessary.
	
	\emph{Sufficiency}: If $n_i\in R_{0\rightarrow j}$, $R_{0\rightarrow j}$ is a combination of $R_{0\rightarrow i}$ and $R_{i\rightarrow j}$. This means $R_{i\rightarrow j}$ exists and thus $n_i$ is forward connected to $n_j$.
\end{proof}
\begin{remark}
	This is a crucial result as it means if one node is forward connected to another, traveling forward from $n_0$ to the latter would require passing through the former. This property is ensured by the scheme of numbering of nodes and the acyclic nature of the graph (corresponding to radial connection of inductors in a circuit).
\end{remark}

\section{Properties of $\mathbf{A}$}\label{sec:prop_of_A}
\begin{definition}\label{definition:3_1}
	For a graph of $N$ nodes, $\mathbf{A}$ is an $N\times N$ matrix defined as:
	\begin{equation}\nonumber
	\mathbf{A}(i,j) = \left\{
	\begin{array}{ll}
	1 ~~~~~\textrm{if}~i=j\\
	-1 ~~\textrm{if}~ n_i ~\textrm{is forward-adjacent to}~ n_j\\
	0 ~~~~~\textrm{otherwise}
	\end{array}
	\right.
	\end{equation}
\end{definition}
This matrix naturally occurs if one writes the nodal equations of a circuit having radial topology using Kirchoff's current law. For ease of visualization, the graph shown in Fig.~\ref{fig:1} and corresponding $\mathbf{A}$ is shown side by side below. It can easily be observed in (\ref{eq:display_A}) that every column has one "1" and one "-1". Having only one "-1" is explained by Corollary~\ref{corollary:1_2}: more than one node cannot be forward adjacent to a particular node. However, if a column (say the $k$\textsuperscript{th} column) does not have any "-1" entry, it has to be concluded that $n_0$ is forward adjacent to $n_k$. Only the first column happens to be such a column in the considered graph.

\vspace{0.15in}

\begin{minipage}[r]{0.3\textwidth}
    \begin{tikzpicture}
    \draw (0,0)node[ocirc]{} to (0,-1)node[ocirc]{} to (0,-2) to (0,-3) to (0,-4)node[ocirc]{} to (0,-5)node[ocirc]{};
    \draw (0,-2) to (1,-2)node[ocirc]{};
    \draw (0,-2)node[ocirc]{} to (-1,-2)node[ocirc]{} to (-2,-2)node[ocirc]{};
    \draw (0,-3)node[ocirc]{} to (-1,-3) to (-2,-3)node[ocirc]{};
    \draw (-1,-3)node[ocirc]{} to (-1,-4)node[ocirc]{} to (-1,-5)node[ocirc]{};
    \node [right=1pt] at (0,0) {$n_0$};
    \node [right=1pt] at (0,-1) {$n_1$};
    \node [right=1pt] at (0,-2.2) {$n_2$};
    \node [right=1pt] at (0,-3) {$n_3$};
    \node [right=1pt] at (0,-4) {$n_4$};
    \node [right=1pt] at (0,-5) {$n_5$};
    \node [right=1pt] at (1,-2) {$n_6$};
    \node [above=1pt] at (-1,-2) {$n_7$};
    \node [left=1pt] at (-2,-2) {$n_8$};
    \node [above=1pt] at (-1,-3) {$n_9$};
    \node [left=1pt] at (-2,-3) {$n_{10}$};
    \node [left=1pt] at (-1,-4) {$n_{11}$};
    \node [left=1pt] at (-1,-5) {$n_{12}$};
    \filldraw[draw=black,fill=black] (0,-0.6) --+ (0.1,0.2) --+ (-0.1,0.2);
    \filldraw[draw=black,fill=black] (0,-1.6) --+ (0.1,0.2) --+ (-0.1,0.2);
    \filldraw[draw=black,fill=black] (0,-2.6) --+ (0.1,0.2) --+ (-0.1,0.2);
    \filldraw[draw=black,fill=black] (0,-3.6) --+ (0.1,0.2) --+ (-0.1,0.2);
    \filldraw[draw=black,fill=black] (0,-4.6) --+ (0.1,0.2) --+ (-0.1,0.2);
    \filldraw[draw=black,fill=black] (0.75,-2) --+ (-0.2,0.1) --+ (-0.2,-0.1);
    \filldraw[draw=black,fill=black] (-0.6,-2) --+ (0.2,0.1) --+ (0.2,-0.1);
    \filldraw[draw=black,fill=black] (-1.6,-2) --+ (0.2,0.1) --+ (0.2,-0.1);
    \filldraw[draw=black,fill=black] (-0.6,-3) --+ (0.2,0.1) --+ (0.2,-0.1);
    \filldraw[draw=black,fill=black] (-1.6,-3) --+ (0.2,0.1) --+ (0.2,-0.1);
    \filldraw[draw=black,fill=black] (-1,-3.6) --+ (0.1,0.2) --+ (-0.1,0.2);
    \filldraw[draw=black,fill=black] (-1,-4.6) --+ (0.1,0.2) --+ (-0.1,0.2);
    \end{tikzpicture}
\end{minipage}
\begin{minipage}[c]{0.65\textwidth}
     \begin{equation}\label{eq:display_A}
        \mathbf{A}=
        \begin{tikzpicture}[baseline=-0.8ex, font=\fontsize{0.105in}{0.1in}\selectfont, every left delimiter/.style={xshift=0.5em},every right delimiter/.style={xshift=-0.5em}]
        \matrix (m) [matrix of math nodes,left delimiter={[},right delimiter={]},row sep=0.7em,column sep=-0.3em,text height=0.5ex, text depth=0.2ex]{
            1 & -1 & 0 & 0 & 0 & 0 & 0 & 0 & 0 & 0 & 0 & 0\\[-5pt]
         	0 & 1 & -1 & 0 & 0 & -1 & -1 & 0 & 0 & 0 & 0 & 0\\[-5pt]
         	0 & 0 & 1 & -1 & 0 & 0 & 0 & 0 & -1 & 0 & 0 & 0\\[-5pt]
         	0 & 0 & 0 & 1 & -1 & 0 & 0 & 0 & 0 & 0 & 0 & 0\\[-5pt]
         	0 & 0 & 0 & 0 & 1 & 0 & 0 & 0 & 0 & 0 & 0 & 0\\[-5pt]
         	0 & 0 & 0 & 0 & 0 & 1 & 0 & 0 & 0 & 0 & 0 & 0\\[-5pt]
         	0 & 0 & 0 & 0 & 0 & 0 & 1 & -1 & 0 & 0 & 0 & 0\\[-5pt]
         	0 & 0 & 0 & 0 & 0 & 0 & 0 & 1 & 0 & 0 & 0 & 0\\[-5pt]
         	0 & 0 & 0 & 0 & 0 & 0 & 0 & 0 & 1 & -1 & -1 & 0\\[-5pt]
         	0 & 0 & 0 & 0 & 0 & 0 & 0 & 0 & 0 & 1 & 0 & 0\\[-5pt]
         	0 & 0 & 0 & 0 & 0 & 0 & 0 & 0 & 0 & 0 & 1 & -1\\[-5pt]
         	0 & 0 & 0 & 0 & 0 & 0 & 0 & 0 & 0 & 0 & 0 & 1\\[-5pt]
         };
        \end{tikzpicture}
     \end{equation}
\end{minipage}

\vspace{0.1in}

\begin{lemma}\label{lemma:2}
	$\mathbf{A}$ is upper-triangular.
	\begin{proof}
		According to Definition~\ref{definition:3_1}, "$1$" can occur only as a diagonal entry. So, lower-triangular entries are "$0$" or "$-1$". Let a lower-triangular entry be "$-1$", i.e., consider $\mathbf{A}(i,j)=-1$ where $i>j$. This would imply $n_i$ is forward adjacent to $n_j$ with $i>j$ which is a contradiction to Definition~\ref{definition:2_2}. Thus, $\mathbf{A}(i,j)=0$ if $i>j$.
	\end{proof}
\end{lemma}
\begin{corollary}\label{corollary:2_1}
	$|\mathbf{A}|=1$.
	\begin{proof}
		$\mathbf{A}$ is upper-triangular. Its eigenvalues, i.e., the diagonal entries are all unity according to the Definition~\ref{definition:3_1}. Therefore, $|\mathbf{A}|$, which is the product of eigenvalues of $\mathbf{A}$, is also unity.
	\end{proof}
\end{corollary}

\begin{lemma}\label{lemma:3}
	Each column of $\mathbf{A}$ has at most two non-zero entries: one is the diagonal entry which is unity while the other is $-1$.
	\begin{proof}
		Consider the $j$\textsuperscript{th} column. This column by definition contains $\mathbf{A}(j,j)=1$. Let there be two other non-zero entries, viz., $\mathbf{A}(i,j)=\mathbf{A}(i',j)=-1$ where $i,i'<j$. This means that there are two distinct nodes $n_i$ and $n_{i'}$ such that both are forward adjacent to $n_j$. This is a contradiction to Corollary~\ref{corollary:1_2} and hence inadmissible.
	\end{proof}
\end{lemma}
\begin{remark}
	A column of $\mathbf{A}$ may have only one non-zero entry corresponding to a node that is adjacent to $n_0$ as already discussed before.
\end{remark}

\begin{definition}\label{definition:3_2}
	The $j$\textsuperscript{th} column of a matrix $\mathbf{Q}$ is denoted by $C^{\mathbf{Q}}_j$. The $i$\textsuperscript{th} entry of this column is denoted as $C^{\mathbf{Q}}_j(i)$.
	\begin{remark}
		Observe that, $\mathbf{Q}(i,j)=C^{\mathbf{Q}}_j(i)$
	\end{remark}
\end{definition}
\begin{definition}\label{definition:3_3}
	$X_i^N$ is an $N$-dimensional column vector with "1" as its $i$\textsuperscript{th} entry and the rest are zero.
\end{definition}
\begin{remark}
	Definition~\ref{definition:3_3} allows succinctly writing the columns of $\mathbf{A}$. For instance, the $9$\textsuperscript{th} column of $\mathbf{A}$ shown in (\ref{eq:display_A}) can now be written as $X_9^{12}-X_3^{12}$. Similarly, every column of $\mathbf{A}$ can be written with two terms since, according to Lemma~\ref{lemma:3}, every column has at most two non-zero entries. For instance, if $n_i$ is forward adjacent to $n_j$, the $j$\textsuperscript{th} column of $\mathbf{A}$ should be $X_j^N-X_i^N$. However, if $n_0$ is forward-adjacent to $n_k$, $C_k^{\mathbf{A}}=X_k^N$. Therefore, $X_0^N$ is defined to be the $N$-dimensional null vector for consistency. (Note: $C_k^{\mathbf{A}}=X_k^N-X_{k'}^N \Rightarrow k'<k$)
\end{remark}

\vspace{0.1in}

Few important facts immediately follow the definition of $X_i^N$ since they are the standard basis vectors in linear algebra. These facts are mentioned below without providing proofs.
\begin{itemize}
	\item \emph{Fact-I}: $X_i^N=X_j^N$ if and only if $i=j$.
	\item \emph{Fact-II}: For $S=\{k_1, ~k_2, ~..., ~k_p\}$ where each $k_i$ is distinct, the set $S_X=\{X_{k_1}^N, ~X_{k_2}^N, ~..., ~X_{k_p}^N\}$ consists of $p$ linearly independent vectors.
	\begin{example}
		For instance, consider $S=\{1,2,4\}$ and $N=5$. So, $X^5_1=[1~0~0~0~0]^T$, $X^5_2=[0~1~0~0~0]^T$ and $X^5_4=[0~0~0~1~0]^T$ are all independent/distinct vectors.
	\end{example}
	\item \emph{Fact-III}: Let $S$ be a set of $p$ distinct positive integers $\{k_1, ~k_2, ~..., ~k_p\}$ and $S'$ be a set of $p$ positive integers $S'=\{k_1', ~k_2', ~..., ~k_p'\}$ where $k_i'$ are not necessarily distinct. If $\displaystyle \sum_{S}X_{k_i}^N=\sum_{S'}X_{k_i'}^N$, then $S=S'$. In other words, if the sum of $p$ distinct $X_i^N$ vectors is same as the sum of a different set of $p$ not necessarily distinct $X_i^N$ vectors, each vector of one set is equal to one vector of the other and thus, obviously, $p=q$ and the integers of the second set are distinct.
	\begin{example}
		Let, $S=\{1,2,4\}$, $S'=\{a,b,c\}$ and $N=5$. Now, if $X^5_1+X^5_2+X^5_4=[1~1~0~1~0]^T=X^5_a+X^5_b+X^5_c$, we must have $\{a,b,c\}=\{1,2,4\}$.
	\end{example}
\end{itemize}

\begin{theorem}\label{theorem:4}
	Node $n_i$ is forward-connected to node-$n_j$ if and only if there exists an additive combination of $k_1$\textsuperscript{th}, $k_2$\textsuperscript{th}, $\cdots$, $k_p$\textsuperscript{th} and $j$\textsuperscript{th} columns of $\mathbf{A}$ such that  $C^{\mathbf{A}}_{k_1}+C^{\mathbf{A}}_{k_2}+\cdots+C^{\mathbf{A}}_{k_p}+C^{\mathbf{A}}_j=X_j^N-X_i^N$ where $i<k_1<k_2<...<k_p<j$.
	\begin{proof}
		\emph{Necessity:} $n_i$ is forward connected to $n_j$. This means either $n_i$ is forward-adjacent to $n_j$, i.e., $C^{\mathbf{A}}_{j}=X_j^N-X_i^N$, in which case the condition is proved; or there exists a sequence of forward-connected nodes in between $n_i$ and $n_j$ that creates the route $R_{i\rightarrow j}$. For the latter case, nodes $n_{k_1},~n_{k_2},~...,~n_{k_p}$ exist with $i<k_1<k_2<...<k_p<j$ such that
		\begin{enumerate}
			\item $n_i$ is forward adjacent to $n_{k_1}$ $~\Rightarrow C^{\mathbf{A}}_{k_1}=X_{k_1}^N-X_{i}^N$
			\item each $n_{k_\beta}$ is forward-adjacent to $n_{k_{\beta+1}} ~\forall~ 1<\beta<p-1 ~~\Rightarrow C^{\mathbf{A}}_{k_{\beta+1}}=X_{k_{\beta+1}}^N-X_{k_{\beta}}^N$
			\item $n_{k_p}$ is forward adjacent to $n_j$ $~\Rightarrow C^{\mathbf{A}}_j=X_{j}^N-X_{k_p}^N$
		\end{enumerate}
	(Note: $k_1$, $k_2$, ..., $k_p$ are not necessarily consecutive integers.) Now, clearly the sum $C^{\mathbf{A}}_{k_1}+C^{\mathbf{A}}_{k_2}+\cdots+C^{\mathbf{A}}_{k_p}+C^{\mathbf{A}}_j$ is equal to $X_j^N-X_i^N$. Thus, the condition is necessary.
	
	\emph{Sufficiency:} There exists an additive combination $C^{\mathbf{A}}_{k_1}+C^{\mathbf{A}}_{k_2}+\cdots+C^{\mathbf{A}}_{k_p}+C^{\mathbf{A}}_j$ with $i<k_1<k_2<...<k_p<j$ such that the sum is $X_j^N-X_i^N$. It is required to prove that $n_i$ is forward connected to $n_j$. Let, $C^{\mathbf{A}}_{k_1}=X_{k_1}^N-X_{k_1'}^N$, $C^{\mathbf{A}}_{k_2}=X_{k_2}^N-X_{k_2'}^N$, ..., $C^{\mathbf{A}}_{k_p}=X_{k_p}^N-X_{k_p'}^N$, $C^{\mathbf{A}}_j=X_j^N-X_{j'}^N$ (this consideration is based on Lemma~\ref{lemma:3}). Now,
	\begin{eqnarray}\nonumber
	X_j^N-X_i^N &=& C^{\mathbf{A}}_{k_1}+C^{\mathbf{A}}_{k_2}+\cdots+C^{\mathbf{A}}_{k_p}+C^{\mathbf{A}}_j\\\nonumber
	&=& (X_{k_1}^N-X_{k_1'}^N)+(X_{k_2}^N-X_{k_2'}^N)+...+(X_{k_p}^N-X_{k_p'}^N)+(X_j^N-X_{j'}^N)\\\nonumber
	&=& (X_{k_1}^N+X_{k_2}^N+...+X_{k_p}^N)+X_j^N-(X_{k_1'}^N+X_{k_2'}^N+...+X_{k_p'}^N+X_{j'}^N)\\\label{eq_aux_2}
	\Rightarrow X_{k_1'}^N+X_{k_2'}^N+...+X_{k_p'}^N+X_{j'}^N &=& (X_{k_1}^N+X_{k_2}^N+...+X_{k_p}^N)+X_i^N
	\end{eqnarray}
	In (\ref{eq_aux_2}), it is easily seen that both sides contain $p+1$ vectors and vectors in the right hand side are mutually independent. Thus, by Fact-III, each column in the left hand side should be equal to one in the right \textcolor{blue}{side}.
	
	Now consider $X_{k_1'}$. $C_{k_1}^{\mathbf{A}}=X_{k_1}-X_{k_1'}$ and therefore $k_1'\neq k_1$. If $X_{k_1'}= X_{k_\beta}$ where $\beta>1$, this would imply $k_1'=k_\beta>k_1$ which is a contradiction since $k_1<k_2<...<k_p$. Therefore, we must have $X_{k_1'}=X_i^N$. Following the same logic, $X_{k_\beta'}=X_{k_{\beta-1}}^N ~\forall~\beta=2,...,p$. For the last element, $X_{j'}^N=X_{k_p}^N$.
	
	Therefore, $n_i$ is forward adjacent to $n_{k_1}$, $n_{k_\beta}$ is forward adjacent to $n_{k_{\beta+1}}~\forall~\beta=1,2,...,p-1$ and $n_{k_p}$ is forward adjacent to $n_j$. Thus, $n_i$ is forward connected to $n_j$.
	\end{proof}
\end{theorem}
\begin{remark}
	It can also be proved that this additive combination to express $X_j^N-X_i^N$ is unique and each participating column $C_{k_\beta}^{\mathbf{A}}$ represents a node $n_{k_\beta}$ such that $n_{k_\beta}\in R_{i\rightarrow j}$. However, these is not required for the present task and therefore their proofs are not discussed.
\end{remark}
\begin{example}
	An example is considered to explain Theorem~\ref{theorem:4}. Consider $n_2$ which is forward connected to $n_{11}$ in Fig.~\ref{fig:1} for instance. Now, there are columns $C_3^{\mathbf{A}}$ and $C_9^{\mathbf{A}}$ such that $C_3^{\mathbf{A}}+C_9^{\mathbf{A}}+C_{11}^{\mathbf{A}}=X_{11}^{12}-X_{2}^{12}$. Further, observe that $n_3$ and $n_9$ are encountered if one travels forward from $n_2$ to $n_{11}$.
\end{example}

\begin{corollary}\label{corollary:4_1}
	If $n_i$ is not forward connected to $n_j$, there exists an $n_{i'}$ where $i'<i$ such that  $C^{\mathbf{A}}_{k_1}+C^{\mathbf{A}}_{k_2}+\cdots+C^{\mathbf{A}}_{k_p}+C^{\mathbf{A}}_j=X_j^N-X_{i'}^N$ where $i<k_1<k_2<...<k_p<j$.
	\begin{proof}
		If $n_i$ is not forward connected to $n_j$, i.e., $n_i \notin R_{0\rightarrow j}$, there exists a series of forward adjacent nodes $n_{h_1}$, $n_{h_2}$, ..., $n_{h_q}$, $n_{k_1}$, $n_{k_2}$, ..., $n_{k_p}$ such that $0<h_1<h_2<...<h_q<i<k_1<k_2<...<k_p<j$ to construct the route  $R_{0\rightarrow j}$. So, there are columns $C^{\mathbf{A}}_{k_1}=X_{k_1}^{N}-X_{h_q}^{N}$, $C^{\mathbf{A}}_{k_2}=X_{k_2}^{N}-X_{k_1}^{N}$, ..., $C^{\mathbf{A}}_j=X_{j}^{N}-X_{k_p}^{N}$. Sum of these column is $C^{\mathbf{A}}_{k_1}+C^{\mathbf{A}}_{k_2}+\cdots+C^{\mathbf{A}}_{k_p}+C^{\mathbf{A}}_j=X_j^N-X_{h_q}^N$ where $h_q<i<k_1<k_2<...<k_p<j$. Obviously, here $h_q=i'$.
	\end{proof}
\end{corollary}

\section{Properties of $\mathbf{A}^{-1}$}\label{sec:prop_of_A-1}
Properties of $\mathbf{A}^{-1}$ that immediately follow from the definition of $\mathbf{A}$ are:
\begin{itemize}
	\item Since $\mathbf{A}$ is upper triangular, $\mathbf{A}^{-1}$ is also upper-triangular.
	\item Eigenvalues of $\mathbf{A}$ are all unity. Therefore, eigenvalues of $\mathbf{A}^{-1}$, i.e., the diagonal entries are also unity.
	\item Since $|\mathbf{A}|=1$,
	\begin{equation}\label{eq:A_cofactor}
	\mathbf{A}^{-1}=(\textrm{cofactor of}~\mathbf{A})^T
	\end{equation}
\end{itemize}

Now, let us focus on the super-diagonal entries of $\mathbf{A}^{-1}$, i.e., $\mathbf{A}^{-1}(i,j)$ where $i<j$. These are obtained from pertinent minors of $\mathbf{A}$.
\begin{definition}\label{definition:4_1}
	$\mathbf{M}_{i,j}$ denotes the matrix whose determinant yields the $(i,j)$\textsuperscript{th} minor of $\mathbf{A}$
\end{definition}
Following (\ref{eq:A_cofactor}), it is seen that
\begin{equation}\label{eq:A_minor}
\mathbf{A}^{-1}(i,j)=(-1)^{i+j}|\mathbf{M}_{j,i}|
\end{equation}
Matrix $\mathbf{M}_{j,i}$ is discussed next.

\section{The matrix $\mathbf{M}_{j,i}$}\label{sec:the_matrix_Mji}
Before proceeding further, we extend the definition of an upper-triangular matrix to rectangular matrices and even columns, as they would be extensively used hereafter.
\begin{definition}\label{definition:5_1}
	For a matrix $\mathbf{Q}_{m\times n}$, an entry $\mathbf{Q}(i,j)$ is (i) a diagonal entry if $i=j$, (ii) a sub-diagonal entry if $i=j+1$, (iii) an upper-triangular entry if $i\geq j$ or (iv) a lower-triangular entry if $i>j$.
\end{definition}
\begin{remark}
	Terms defined in Definition~\ref{definition:5_1} are well-known for square matrices. Their definitions are extended in this article to even include rectangular matrices, rows ($m=1$) and columns ($n=1$).
\end{remark}
\begin{definition}\label{definition:5_2}
	A matrix $\mathbf{Q}_{m\times n}$ is called upper-triangular if $\mathbf{Q}(i,j)=0~\forall~i>j$.
\end{definition}
\begin{remark}
	Observe that $\mathbf{Q}$ becomes simply a column if $n=1$. Thus Definition~\ref{definition:5_2} is applicable to columns as well. The $j$\textsuperscript{th} column of a matrix  $\mathbf{Q}$, i.e., $C_{j}^{\mathbf{Q}}$ is called upper-triangular if $\mathbf{Q}(i,j)=C_{j}^{\mathbf{Q}}(i)=0~\forall~i>j$.
\end{remark}

Critical properties of $\mathbf{M}_{j,i}$ are discussed in this section. For ease of understanding, a specific entry $\mathbf{M}_{8,4}$ is used as an example. Note that, according to (\ref{eq:A_minor}), $\mathbf{A}^{-1}(4,8)=|\mathbf{M}_{8,4}|$ where $\mathbf{M}_{8,4}$ is formed by deleting the shaded row and column, i.e., 8\textsuperscript{th} row and 4\textsuperscript{th} column of $\mathbf{A}$ as depicted in Fig.~\ref{fig:3}. This example represents deletion of the $j$\textsuperscript{th} row and $i$\textsuperscript{th} column of $\mathbf{A}$ in general.
\begin{figure}[h]
\centering
\begin{equation*}
\begin{tikzpicture}[baseline=-0.8ex, font=\fontsize{0.105in}{0.1in}\selectfont, every left delimiter/.style={xshift=0.5em},every right delimiter/.style={xshift=-0.5em}]
\matrix (m) [matrix of math nodes,left delimiter={[},right delimiter={]},row sep=0.7em,column sep=-0.3em,text height=0.5ex, text depth=0.2ex]{
    1 & -1 & 0 & 0 & 0 & 0 & 0 & 0 & 0 & 0 & 0 & 0\\[-5pt]
	0 & 1 & -1 & 0 & 0 & -1 & -1 & 0 & 0 & 0 & 0 & 0\\[-5pt]
	0 & 0 & 1 & -1 & 0 & 0 & 0 & 0 & -1 & 0 & 0 & 0\\[-5pt]
	0 & 0 & 0 & 1 & -1 & 0 & 0 & 0 & 0 & 0 & 0 & 0\\[-5pt]
	0 & 0 & 0 & 0 & 1 & 0 & 0 & 0 & 0 & 0 & 0 & 0\\[-5pt]
	0 & 0 & 0 & 0 & 0 & 1 & 0 & 0 & 0 & 0 & 0 & 0\\[-5pt]
	0 & 0 & 0 & 0 & 0 & 0 & 1 & -1 & 0 & 0 & 0 & 0\\[-5pt]
	0 & 0 & 0 & 0 & 0 & 0 & 0 & 1 & 0 & 0 & 0 & 0\\[-5pt]
	0 & 0 & 0 & 0 & 0 & 0 & 0 & 0 & 1 & -1 & -1 & 0\\[-5pt]
	0 & 0 & 0 & 0 & 0 & 0 & 0 & 0 & 0 & 1 & 0 & 0\\[-5pt]
	0 & 0 & 0 & 0 & 0 & 0 & 0 & 0 & 0 & 0 & 1 & -1\\[-5pt]
	0 & 0 & 0 & 0 & 0 & 0 & 0 & 0 & 0 & 0 & 0 & 1\\[-5pt]
};
\begin{pgfonlayer}{background}
\fill [Gray] (-1.6,2.6) rectangle (-1.15,-2.5);
\fill [Gray] (-2.95,-0.31) rectangle (2.9,-0.75);
\end{pgfonlayer}
\end{tikzpicture}
\xrightarrow[\text{$j$\textsuperscript{th} row deleted}]{\text{$i$\textsuperscript{th} column deleted}}
\begin{tikzpicture}[baseline=-0.8ex, font=\fontsize{0.105in}{0.1in}\selectfont, every left delimiter/.style={xshift=0.5em},every right delimiter/.style={xshift=-0.5em}]
\matrix (m) [matrix of math nodes,left delimiter={[},right delimiter={]},row sep=0.7em,column sep=-0.3em,text height=0.5ex, text depth=0.2ex]{
	1 & -1 & 0 & 0 & 0 & 0 & 0 & 0 & 0 & 0 & 0\\[-5pt]
	0 & 1 & -1 & 0 & -1 & -1 & 0 & 0 & 0 & 0 & 0\\[-5pt]
	0 & 0 & 1 & 0 & 0 & 0 & 0 & -1 & 0 & 0 & 0\\[-5pt]
	0 & 0 & 0 & -1 & 0 & 0 & 0 & 0 & 0 & 0 & 0\\[-5pt]
	0 & 0 & 0 & 1 & 0 & 0 & 0 & 0 & 0 & 0 & 0\\[-5pt]
	0 & 0 & 0 & 0 & 1 & 0 & 0 & 0 & 0 & 0 & 0\\[-5pt]
	0 & 0 & 0 & 0 & 0 & 1 & -1 & 0 & 0 & 0 & 0\\[-5pt]
	0 & 0 & 0 & 0 & 0 & 0 & 0 & 1 & -1 & -1 & 0\\[-5pt]
	0 & 0 & 0 & 0 & 0 & 0 & 0 & 0 & 1 & 0 & 0\\[-5pt]
	0 & 0 & 0 & 0 & 0 & 0 & 0 & 0 & 0 & 1 & -1\\[-5pt]
	0 & 0 & 0 & 0 & 0 & 0 & 0 & 0 & 0 & 0 & 1\\[-5pt]
};
\draw[dashed] (-1.4,-2.15) to (-1.4,2.3);
\draw[dashed] (0.65,-2.15) to (0.65,2.3);
\draw[dashed] (-2.7,-0.53) to (2.7,-0.53);
\begin{pgfonlayer}{background}
\fill [Gray] (-2.7,2.3) rectangle (-2.35,1.89);
\fill [Gray] (-2.35,1.89) rectangle (-1.9,1.48);
\fill [Gray] (-1.9,1.48) rectangle (-1.45,1.07);
\fill [Gray] (-1.35,1.07) rectangle (-0.9,0.66);
\fill [Gray] (-0.9,0.66) rectangle (-0.4,0.25);
\fill [Gray] (-0.4,0.25) rectangle (0.1,-0.16);
\fill [Gray] (0.1,-0.16) rectangle (0.6,-0.48);
\fill [Gray] (0.7,-0.57) rectangle (1.2,-0.98);
\fill [Gray] (1.2,-0.98) rectangle (1.7,-1.39);
\fill [Gray] (1.7,-1.39) rectangle (2.2,-1.8);
\fill [Gray] (2.2,-1.8) rectangle (2.7,-2.21);
\end{pgfonlayer}
\draw[->,shorten <=1pt] (-1.35,2.5) -- (2,2.5)node[right] {$y\geq i$};
\draw[->,shorten <=1pt] (-1.45,2.5) -- (-2,2.5)node[left] {$y< i$};
\draw[->,shorten <=1pt] (0.7,-2.4) -- (1.9,-2.4)node[right] {$y\geq j$};
\draw[->,shorten <=1pt] (2.95,-0.48) -- (2.95,2.2)node[right] {$x<j$};
\draw[->,shorten <=1pt] (2.95,-0.58) -- (2.95,-2.15)node[right] {$x\geq j$};
\end{tikzpicture}
\end{equation*}
\caption{\label{fig:3} \emph{Depiction of forming $\mathbf{M}_{j,i}$ from $\mathbf{A}$}: The shaded row and column of $\mathbf{A}$ (on the left) are the $j$\textsuperscript{th} row and the $i$\textsuperscript{th} column respectively. Diagonal entries of the resulting $\mathbf{M}_{j,i}$ (on the right) are highlighted.}
\end{figure}
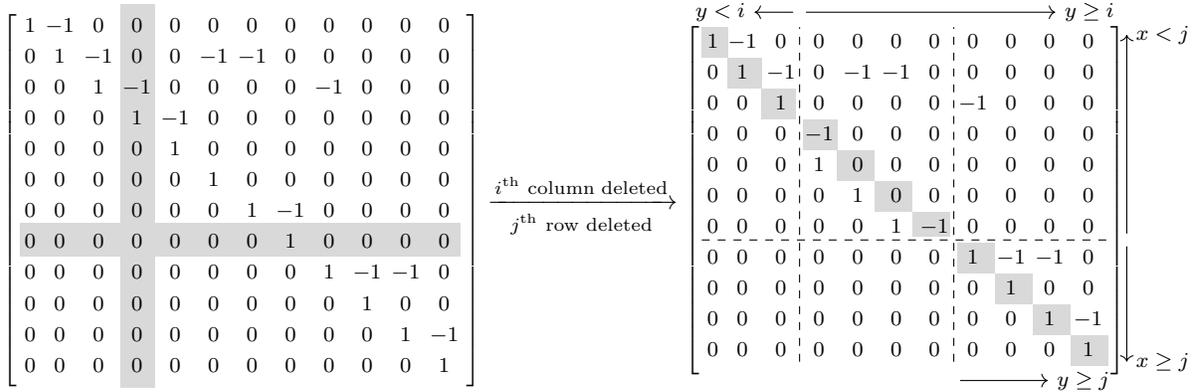

A careful inspection of Fig.~\ref{fig:3} reveals that entries of $\mathbf{M}_{j,i}(x,y)$ can be written as
\begin{eqnarray}\label{eq_2_lemma_row_column_deletion_1}
\mathbf{M}_{j,i}(x,y) &=& \mathbf{A}(x,y)~\textrm{if}~x<j,y<i \\\label{eq_2_lemma_row_column_deletion_2}
\mathbf{M}_{j,i}(x,y) &=& \mathbf{A}(x,y+1)~\textrm{if}~x<j,y\geq i\\\label{eq_2_lemma_row_column_deletion_3}
\mathbf{M}_{j,i}(x,y) &=& \mathbf{A}(x+1,y)~\textrm{if}~x\geq j,y<i~~( =0 ~\textrm{as}~x>y)\\\label{eq_2_lemma_row_column_deletion_4}
\mathbf{M}_{j,i}(x,y) &=& \mathbf{A}(x+1,y+1)~\textrm{if}~x\geq j,y\geq i
\end{eqnarray}

It can be seen in Fig.~\ref{fig:3} that
\begin{itemize}
  \item First three columns (before the $i$\textsuperscript{th} column) and last four columns (after the $j$\textsuperscript{th} column) are upper-triangular with diagonal entries as unity.
  \item As $C_4^{\mathbf{A}}$ is deleted, the following columns move towards left. Consequently, the diagonal entries of $C_5^{\mathbf{A}}$, $C_6^{\mathbf{A}}$, and $C_7^{\mathbf{A}}$ become sub-diagonal entries in $C_4^{\mathbf{M}_{8,4}}$, $C_5^{\mathbf{M}_{8,4}}$, and $C_6^{\mathbf{M}_{8,4}}$. These columns have lost their upper-diagonal property.
  \item $C_7^{\mathbf{M}_{8,4}}$ is a unique column in $\mathbf{M}_{8,4}$. This column does not have a unity entry as it got deleted with the deletion of the $8$\textsuperscript{th} column of $\mathbf{A}$. It may also be noted that this column is upper-triangular.
\end{itemize}
Before proceeding further, let us put our objective in perspective. We are interested in the structure of $\mathbf{A}^{-1}$ and that, according to \eqref{eq:A_minor}, is fully understood from the knowledge of a general $|\mathbf{M}_{j,i}|$. Finding the determinant of $\mathbf{M}_{j,i}$, however, becomes tricky for a general case because $\mathbf{M}_{j,i}$ is not an upper-triangular matrix as is evident from Fig.~\ref{fig:3}. Inspection of Fig.~\ref{fig:3} reveals that not all the columns are responsible for destroying the upper-triangular nature and those troublesome columns are present in the middle. The aforementioned properties of the columns are carefully listed. One can identify some more interesting facts in this specific example, but they would not be necessarily true for any arbitrary graph while the listed ones are. Hereafter we shall prove these properties from the definition of $\mathbf{A}$ and $\mathbf{M}_{j,i}$. This is necessary as these properties are going to be made use of to find $|\mathbf{M}_{j,i}|$ for a generic case. Each of the following subsections contains its summary at the end.

\subsection{Columns on either side}\label{subsec:columns_side}
\begin{lemma}\label{lemma:5}
	If $\mathbf{A}$ is of dimension $N\times N$ and $i<j$, first $i-1$ and last $N-j$ columns of $\mathbf{M}_{j,i}$, i.e., $C_k^{\mathbf{M}_{j,i}}~ \forall k<i, k\geq j$ are upper-triangular with the diagonal entries as unity.
	\begin{proof}
		\emph{First $i-1$ columns of $\mathbf{M}_{j,i}$}: Consider $\mathbf{M}_{j,i}(x,y)$ with $x\leq y$ as an arbitrary upper-triangular entry of these columns. Now,
		\begin{eqnarray}
		\mathbf{M}_{j,i}(x,y)~\textrm{for}~x\leq y, ~y<i ~~~\textrm{[$y<i$ since we are considering first $i-1$ columns]}\\\nonumber
		= \mathbf{M}_{j,i}(x,y)~\textrm{for}~x\leq y<i<j ~~~\textrm{[since $\mathbf{M}_{j,i}$ is defined for $i<j$]}\\\nonumber
		= \mathbf{A}(x,y) ~~~\textrm{[according to (\ref{eq_2_lemma_row_column_deletion_1})]}~~~~~~~~~~~~~~~~~~~~~~~~~~~~~~~~~~~~~~~~~~~
		\end{eqnarray}
		Now, consider an arbitrary lower-triangular entry in these columns of $\mathbf{M}_{j,i}$, i.e., $\mathbf{M}_{j,i}(x,y)$ with $x>y$. According to (\ref{eq_2_lemma_row_column_deletion_1}) and ((\ref{eq_2_lemma_row_column_deletion_3})), these entries can be written as
		\begin{equation}
		\mathbf{M}_{j,i}(x,y)~\textrm{for}~x>y, ~y<i<j=
		\left\{
		\begin{array}{ll}
		[\mathbf{A}(x,y)] \textrm{~for~} y<x<j ~~\textrm{[according to (\ref{eq_2_lemma_row_column_deletion_1})]}=0 ~~\textrm{[since $x>y$]}\\[2pt]
		0 ~\textrm{for}~x\geq j, ~y<i<j ~~\textrm{[according to (\ref{eq_2_lemma_row_column_deletion_3})]}
		\end{array}
		\right.
		\end{equation}
		Therefore, the upper-triangular entries of the first $i-1$ columns of $\mathbf{M}_{j,i}$ is exactly the same as that of [$\mathbf{A}$] and the lower-triangular entries are zero. Hence, the first $i-1$ columns are upper-triangular with diagonal entries unity.
		
		\emph{Last $N-j$ columns of $\mathbf{M}_{j,i}$}: Now, focus on the last $N-j$ columns. An arbitrary lower-triangular entry in these columns is $\mathbf{M}_{j,i}(x,y)$ with $x>y$. Now,
		\begin{eqnarray}\nonumber
		\mathbf{M}_{j,i}(x,y)~\textrm{for}~x>y\geq j ~~~\textrm{[$y\geq j$ since we are considering last  $N-j$ columns]}\\
		= \mathbf{A}(x+1,y+1) [\textrm{according to}~(\ref{eq_2_lemma_row_column_deletion_4})]	~~= 0 ~\textrm{[since $x+1>y+1$]}
		\end{eqnarray}
		Now, consider a diagonal entry in these columns of $\mathbf{M}_{j,i}$, i.e., $\mathbf{M}_{j,i}(x,y)$ with $x=y\geq j$.
		\begin{equation}
		\mathbf{M}_{j,i}(x,y)~\textrm{for}~x=y\geq j = \mathbf{A}(x+1,y+1) ~~\textrm{[according to  (\ref{eq_2_lemma_row_column_deletion_4})]}=1
		\end{equation}
		Therefore, the lower-triangular entries of the last $N-j$ columns of $\mathbf{M}_{j,i}$ are zero, i.e., these columns are upper-triangular and diagonal entries are unity.
	\end{proof}
\end{lemma}

\emph{\textbf{Summary}}: The side-columns are all upper-triangular with diagonal entries as unity. Therefore, we do not need to focus on these columns while determining $|\mathbf{M}_{j,i}|$.

\subsection{Columns in the middle}\label{subsec:columns_middle}
Crucial properties of columns of $\mathbf{M}_{j,i}$ enclosed by the two vertical dashed lines in Fig.~\ref{fig:3} are stated and proved in the following two theorems. Their significance is briefly mentioned in the summary of this subsection. The last column within the dashed line, i.e. $C^{\mathbf{M}_{j,i}}_{j-1}$ for a general graph, is not excluded here and discussed in the next subsection. "\emph{The middle-columns}" shall refer to all the rest of the columns within the dashed lines.
\begin{lemma}\label{lemma:6}
	Sub-diagonal entries of $C_k^{\mathbf{M}_{j,i}}$, i.e., $\mathbf{M}_{j,i}(k+1,k)$ where $i\leq k<j-1$ are unity. Other lower-triangular entries in these columns are zero.
	\begin{proof}
		Consider the sub-diagonal entries first.
		\begin{eqnarray}\nonumber
		\mathbf{M}_{j,i}(k+1,k)~ \forall i\leq k<j-1 &=& \mathbf{M}_{j,i}(k+1,k)~\forall k+1<j,~k\geq i\\
		&=&\mathbf{A}(k+1,k+1)~~\textrm{[according to (\ref{eq_2_lemma_row_column_deletion_2})]} ~= 1
		\end{eqnarray}
		Now  consider the other lower-triangular entries, i.e., $\mathbf{M}_{j,i}(x,k)~\forall x>k+1$.
		\begin{eqnarray}\nonumber
		\mathbf{M}_{j,i}(x,k)~ \forall x>k+1, ~i\leq k<j-1 &=& \mathbf{M}_{j,i}(x,k)~\forall x<j,~y\geq i=\mathbf{A}(x,y+1) [\textrm{according to}~(\ref{eq_2_lemma_row_column_deletion_2})]=0\\
		&=&\mathbf{M}_{j,i}(x,y)~\forall x\geq j,~y\geq i=0 ~~[\textrm{according to}~(\ref{eq_2_lemma_row_column_deletion_4})]
		\end{eqnarray}
	\end{proof}
\end{lemma}

\begin{theorem}\label{theorem:7}
	If $C_k^{\mathbf{A}}=X_k^N-X_{k'}^N$ where $i<k<j$, $C_{k-1}^{\mathbf{M}_{j,i}}=X_k^{N-1}-X_{k'}^{N-1}$.
	\begin{proof}
		Each $(k-1)$\textsuperscript{th} column of $\mathbf{M}_{j,i}$ can be subdivided into two parts: (i) first $j-1$ entries, i.e., $C_{k-1}^{\mathbf{M}_{j,i}}(x)~\forall x<j$ and (ii) the rest. For the first $j-1$ entries
		\begin{eqnarray}\nonumber
		C_{k-1}^{\mathbf{M}_{j,i}}(x)=\mathbf{M}_{j,i}(x,k-1) &=& \mathbf{A}(x,k) ~\forall~x<j~~~~[\textrm{according to (\ref{eq_2_lemma_row_column_deletion_2})}]\\\label{eq_19}
		&=& C_k^{\mathbf{A}}(x) \forall x<j, i<k<j
		\end{eqnarray}
		For rest of the entries
		\begin{equation}
		C_{k-1}^{\mathbf{M}_{j,i}}(x)=\mathbf{A}(x+1,k) ~\forall~x\geq j~~[\textrm{according to (\ref{eq_2_lemma_row_column_deletion_4})}]=0~\textrm{[since $x+1>j>k$]}
		\end{equation}
		The upper-triangular entries of $C_k^{\mathbf{A}}$ are all present in the first $j-1$ entries since $k<j$. According to (\ref{eq_19}), they are all present in $C_{k-1}^{\mathbf{M}_{j,i}}$ in the same row position. Thus the non zero entries of $C_{k-1}^{\mathbf{M}_{j,i}}$ are exactly the same (and also in the same position) as $C_k^{\mathbf{A}}$. Thus, $C_{k-1}^{\mathbf{M}_{j,i}}=X_k^{N-1}-X_{k'}^{N-1}$ if $C_k^{\mathbf{A}}=X_k^N-X_{k'}^N$ where $i<k<j$.
	\end{proof}
\end{theorem}

\begin{corollary}\label{corollary:7_1}
	Consider an addition of columns of $\mathbf{A}$, $C_{k_1}^{\mathbf{A}}+C_{k_2}^{\mathbf{A}}+...+C_{k_p}^{\mathbf{A}}=X_{k_p}^N-X_{k_1'}$ where $i<k_1<k_2<...<k_p<j$. The addition of columns of $\mathbf{M}_{j,i}$, i.e., $C_{k_1-1}^{\mathbf{M}_{j,i}}+C_{k_2-1}^{\mathbf{M}_{j,i}}+...+C_{k_p-1}^{\mathbf{M}_{j,i}}$ would be $X_{k_p}^{N-1}-X_{k_1'}^{N-1}$.
\end{corollary}
\begin{remark}
	The proof is trivial using Theorem~\ref{theorem:7} and therefore omitted.
\end{remark}

\emph{\textbf{Summary}}: Lemma~\ref{lemma:6} shows that the sub-diagonal entries of \emph{the middle-columns}, even in a general case, are unity. This happens because these columns move towards left as the $i$\textsuperscript{th} column of $\mathbf{A}$ is deleted and diagonal entries of these columns while in $\mathbf{A}$ become subdiagonal entries in $\mathbf{M}_{j,i}$. Theorem~\ref{theorem:7} provides a formula to readily obtain these columns from $\mathbf{A}$ and Corollary~\ref{corollary:7_1} shows how the sum of these columns is related to the sum of the parent columns in $\mathbf{A}$. Most importantly, Corollary~\ref{corollary:7_1} puts Theorem~\ref{theorem:4} and its corollary, which was proved for a generic $\mathbf{A}$, in the perspective of $\mathbf{M}_{j,i}$ for a generic case. This is crucial because if $n_i$ is forward connected to $n_j$, $R_{i\rightarrow j}$ will be created by nodes $n_k$ where $i<k<j$. Therefore, the middle columns, i.e. $C_k^{\mathbf{M}_{j,i}}~\forall i\geq k<j$, are crucial as they carry the information regarding if and how $n_i$ is forward connected to $n_j$. Further significance will be evident in the next section.

\subsection{The critical column}\label{subsec:column_critical}
The $(j-1)$\textsuperscript{th} column of $\mathbf{M}_{j,i}$ can at most have one non-zero entry and that will be "-1". This non-zero entry appears at some $\mathbf{M}_{j,i}(m,j-1)$ where $m<j-1$. The entry "$1$" is deleted because the $j$\textsuperscript{th} row of $\mathbf{A}$ is deleted.
\begin{lemma}\label{lemma:8}
	If $C_{j}^{\mathbf{A}}=X_j^N-X_{j'}^N$, $C_{j-1}^{\mathbf{M}_{j,i}}$ must be $-X_{j'}^{N-1}$.
	\begin{proof}
		According to (\ref{eq_2_lemma_row_column_deletion_2}), for $x<j$, $\displaystyle		 C_{j-1}^{\mathbf{M}_{j,i}}(x)=\mathbf{M}_{j,i}(x,j-1)=\mathbf{A}(x,j)= C_j^{\mathbf{A}}(x)$.
		
		For $x\geq j$, according to (\ref{eq_2_lemma_row_column_deletion_3}), $\displaystyle		 C_{j-1}^{\mathbf{M}_{j,i}}(x)=\mathbf{A}(x+1,j)=0$.
		
		It is now clearly seen that the first $j-1$ entries of $C_{j-1}^{\mathbf{M}_{j,i}}$ are equal to the first $j-1$ entries of $C_{j}^{\mathbf{A}}$ and the rest are zero. Thus, $C_j^{\mathbf{M}_{j,i}}=-X_{j'}^{N-1}$.
	\end{proof}
\end{lemma}

\begin{theorem}\label{theorem:9}
	If $C_{k_1}^{\mathbf{A}}+C_{k_2}^{\mathbf{A}}+...+C_{k_p}^{\mathbf{A}}+C_j^{\mathbf{A}}=X_{j}^N-X_{k_1'}^N$ where $i<k_1<k_2<...<k_p<j$ and $C_{k_1}^{\mathbf{A}}=X_{k_1}^N-X_{k_1'}^N$, $C_{k_1-1}^{\mathbf{M}_{j,i}}+C_{k_2-1}^{\mathbf{M}_{j,i}}+...+C_{k_p-1}^{\mathbf{M}_{j,i}}+C_{j-1}^{\mathbf{M}_{j,i}}=-X_{k_1'}^{N-1}$
	\begin{proof}
		Let, $\displaystyle C_j^{\mathbf{A}} = X_j^N-X_{j'}^N$. Substituting this in the given relationship results in
		\begin{equation}\label{eq_theo_9_1}
		C_{k_1}^{\mathbf{A}}+C_{k_2}^{\mathbf{A}}+...+C_{k_p}^{\mathbf{A}}=X_{j'}^N-X_{k_1'}
		\end{equation}
		Now, according to Corollary~\ref{corollary:7_1}, (\ref{eq_theo_9_1}) would imply
		\begin{equation}\label{eq_theo_9_2}
		C_{k_1}^{\mathbf{M}_{j,i}}+C_{k_2}^{\mathbf{M}_{j,i}}+...+C_{k_p}^{\mathbf{M}_{j,i}} = X_{j'}^{N-1}-X_{k_1'}^{N-1}
		\end{equation}
		According to Lemma~\ref{lemma:8},
		\begin{equation}\label{eq_theo_9_3}
		C_j^{\mathbf{A}} = X_j^N-X_{j'}^N \Rightarrow C_{j-1}^{\mathbf{M}_{j,i}}=-X_{j'}^{N-1}
		\end{equation}
		Adding (\ref{eq_theo_9_2}) and (\ref{eq_theo_9_3}), we get, ~~~
		$C_{k_1-1}^{\mathbf{M}_{j,i}}+C_{k_2-1}^{\mathbf{M}_{j,i}}+...+C_{k_p-1}^{\mathbf{M}_{j,i}}+C_{j-1}^{\mathbf{M}_{j,i}}=-X_{k_1'}^{N-1}$
	\end{proof}
\end{theorem}

\emph{\textbf{Summary}}: Lemma~\ref{lemma:8} in this subsection shows how the critical column in $\mathbf{M}_{j,i}$ is related to its parent column in $\mathbf{A}$. Theorem~\ref{theorem:9} uses Corollary~\ref{corollary:7_1} to generate a result which is critical for making all the columns of $\mathbf{M}_{j,i}$ upper-triangular. This upper-triangularization is discussed next.

\section{Upper-triangularization of $\mathbf{M}_{j,i}$}\label{sec:upper_triangularization}

Matrix $\mathbf{M}_{j,i}$, unlike its parent matrix $\mathbf{A}$, is not upper-triangular. Further, as evident from \eqref{eq:A_minor}, knowledge of $|\mathbf{M}_{j,i}|$ is crucial to understand $\mathbf{A}^{-1}$. It is seen in the earlier section that some columns in $\mathbf{M}_{j,i}$ are upper-triangular just like their parent columns in $\mathbf{A}$. In mathematical terms, $C_k^{\mathbf{M}_{j,i}}$ is upper-triangular if $k<i$ or $k\geq j$. These columns were termed as \emph{the side-columns} and for the considered case they appear on either side of the two vertical dashed lines in Fig.~\ref{fig:3}. Further their diagonal entries are unity; therefore, if the columns within the dashed lines are made upper-triangular, one would not have to worry about the side columns. Now we focus on the columns within the dashed lines: \emph{the middle-columns} and \emph{the critical column} and make them upper-triangular without disturbing the side columns.

The middle-columns, i.e. $C_k^{\mathbf{M}_{j,i}}$ for $i\leq k<j-1$, have their sub-diagonal entries as unity and rest of the lower-triangular entries are zero (according to Lemma~\ref{lemma:6}). With some effort, it can be visualised that if somehow these columns can be right shifted, "1" will appear as the diagonal entry and all lower-triangular entries would be zero. \footnote{These columns were left shifted owing to deletion of the $i$\textsuperscript{th} column of $\mathbf{A}$ while $\mathbf{M}_{j,i}$ was being formed. That is why the diagonal entries appeared at the sub-diagonal entries. Reversing that process, just for these columns, will again make them appear as diagonal entries.} In order to achieve that let us move the $j-1$\textsuperscript{th} column, i.e. the critical column, by a series of column interchanges until it occupies the place of the $i$\textsuperscript{th} column as shown in Fig.~\ref{fig:4}. Thus the side-columns are not disturbed, but the resulting matrix will have all the middle columns made upper-triangular except the critical column which now would appear as the $i$\textsuperscript{th} column.

\begin{figure}[h]
  \centering
    \begin{eqnarray*}
    \begin{tikzpicture}[baseline=-0.8ex, font=\fontsize{0.105in}{0.1in}\selectfont, every left delimiter/.style={xshift=0.5em},every right delimiter/.style={xshift=-0.5em}]
    \matrix (m) [matrix of math nodes,left delimiter={[},right delimiter={]},row sep=0.7em,column sep=-0.3em,text height=0.5ex, text depth=0.2ex]{
    	1 & -1 & 0 & 0 & 0 & 0 & 0 & 0 & 0 & 0 & 0\\[-5pt]
    	0 & 1 & -1 & 0 & -1 & -1 & 0 & 0 & 0 & 0 & 0\\[-5pt]
    	0 & 0 & 1 & 0 & 0 & 0 & 0 & -1 & 0 & 0 & 0\\[-5pt]
    	0 & 0 & 0 & -1 & 0 & 0 & 0 & 0 & 0 & 0 & 0\\[-5pt]
    	0 & 0 & 0 & 1 & 0 & 0 & 0 & 0 & 0 & 0 & 0\\[-5pt]
    	0 & 0 & 0 & 0 & 1 & 0 & 0 & 0 & 0 & 0 & 0\\[-5pt]
    	0 & 0 & 0 & 0 & 0 & 1 & -1 & 0 & 0 & 0 & 0\\[-5pt]
    	0 & 0 & 0 & 0 & 0 & 0 & 0 & 1 & -1 & -1 & 0\\[-5pt]
    	0 & 0 & 0 & 0 & 0 & 0 & 0 & 0 & 1 & 0 & 0\\[-5pt]
    	0 & 0 & 0 & 0 & 0 & 0 & 0 & 0 & 0 & 1 & -1\\[-5pt]
    	0 & 0 & 0 & 0 & 0 & 0 & 0 & 0 & 0 & 0 & 1\\[-5pt]
    };
    \draw[dashed] (-1.4,-2.15) to (-1.4,2.3);
    \draw[dashed] (0.65,-2.15) to (0.65,2.3);
    \draw[<->,blue] (0.4,-2.2) to (0.4,-2.3) to [out=270, in=270] (-0.1,-2.3) to [->] (-0.1,-2.2);

    \begin{pgfonlayer}{background}
    \fill [Gray] (0.15,2.3) rectangle (0.6,-2.21);
    \end{pgfonlayer}
    \end{tikzpicture}
    \xrightarrow{\displaystyle{C_7\leftrightarrow C_6}}
    \begin{tikzpicture}[baseline=-0.8ex, font=\fontsize{0.105in}{0.1in}\selectfont, every left delimiter/.style={xshift=0.5em},every right delimiter/.style={xshift=-0.5em}]
    \matrix (m) [matrix of math nodes,left delimiter={[},right delimiter={]},row sep=0.7em,column sep=-0.3em,text height=0.5ex, text depth=0.2ex]{
    	1 & -1 & 0 & 0 & 0 & 0 & 0 & 0 & 0 & 0 & 0\\[-5pt]
    	0 & 1 & -1 & 0 & -1 & -1 & 0 & 0 & 0 & 0 & 0\\[-5pt]
    	0 & 0 & 1 & 0 & 0 & 0 & 0 & -1 & 0 & 0 & 0\\[-5pt]
    	0 & 0 & 0 & -1 & 0 & 0 & 0 & 0 & 0 & 0 & 0\\[-5pt]
    	0 & 0 & 0 & 1 & 0 & 0 & 0 & 0 & 0 & 0 & 0\\[-5pt]
    	0 & 0 & 0 & 0 & 1 & 0 & 0 & 0 & 0 & 0 & 0\\[-5pt]
    	0 & 0 & 0 & 0 & 0 & 1 & -1 & 0 & 0 & 0 & 0\\[-5pt]
    	0 & 0 & 0 & 0 & 0 & 0 & 0 & 1 & -1 & -1 & 0\\[-5pt]
    	0 & 0 & 0 & 0 & 0 & 0 & 0 & 0 & 1 & 0 & 0\\[-5pt]
    	0 & 0 & 0 & 0 & 0 & 0 & 0 & 0 & 0 & 1 & -1\\[-5pt]
    	0 & 0 & 0 & 0 & 0 & 0 & 0 & 0 & 0 & 0 & 1\\[-5pt]
    };
    \draw[dashed] (-1.4,-2.15) to (-1.4,2.3);
    \draw[dashed] (0.65,-2.15) to (0.65,2.3);
    \draw[<->,blue] (-0.1,-2.2) to (-0.1,-2.3) to [out=270, in=270] (-0.6,-2.3) to [->] (-0.6,-2.2);
    \begin{pgfonlayer}{background}
    \fill [Gray] (-0.35,2.3) rectangle (0.15,-2.21);
    \end{pgfonlayer}
    \end{tikzpicture}\\[-7pt]
    \downarrow C_6\leftrightarrow C_5~~~~~\\[-7pt]
    \begin{tikzpicture}[baseline=-0.8ex, font=\fontsize{0.105in}{0.1in}\selectfont, every left delimiter/.style={xshift=0.5em},every right delimiter/.style={xshift=-0.5em}]
    \matrix (m) [matrix of math nodes,left delimiter={[},right delimiter={]},row sep=0.7em,column sep=-0.3em,text height=0.5ex, text depth=0.2ex]{
    	1 & -1 & 0 & 0 & 0 & 0 & 0 & 0 & 0 & 0 & 0\\[-5pt]
    	0 & 1 & -1 & 0 & -1 & -1 & 0 & 0 & 0 & 0 & 0\\[-5pt]
    	0 & 0 & 1 & 0 & 0 & 0 & 0 & -1 & 0 & 0 & 0\\[-5pt]
    	0 & 0 & 0 & -1 & 0 & 0 & 0 & 0 & 0 & 0 & 0\\[-5pt]
    	0 & 0 & 0 & 1 & 0 & 0 & 0 & 0 & 0 & 0 & 0\\[-5pt]
    	0 & 0 & 0 & 0 & 1 & 0 & 0 & 0 & 0 & 0 & 0\\[-5pt]
    	0 & 0 & 0 & 0 & 0 & 1 & -1 & 0 & 0 & 0 & 0\\[-5pt]
    	0 & 0 & 0 & 0 & 0 & 0 & 0 & 1 & -1 & -1 & 0\\[-5pt]
    	0 & 0 & 0 & 0 & 0 & 0 & 0 & 0 & 1 & 0 & 0\\[-5pt]
    	0 & 0 & 0 & 0 & 0 & 0 & 0 & 0 & 0 & 1 & -1\\[-5pt]
    	0 & 0 & 0 & 0 & 0 & 0 & 0 & 0 & 0 & 0 & 1\\[-5pt]
    };
    \draw[dashed] (-1.4,-2.15) to (-1.4,2.3);
    \draw[dashed] (0.65,-2.15) to (0.65,2.3);
    \begin{pgfonlayer}{background}
    \fill [Gray] (-1.35,2.3) rectangle (-0.85,-2.21);
    \end{pgfonlayer}
    \end{tikzpicture}
    \xleftarrow{\displaystyle{C_5\leftrightarrow C_4}}
    \begin{tikzpicture}[baseline=-0.8ex, font=\fontsize{0.105in}{0.1in}\selectfont, every left delimiter/.style={xshift=0.5em},every right delimiter/.style={xshift=-0.5em}]
    \matrix (m) [matrix of math nodes,left delimiter={[},right delimiter={]},row sep=0.7em,column sep=-0.3em,text height=0.5ex, text depth=0.2ex]{
    	1 & -1 & 0 & 0 & 0 & 0 & 0 & 0 & 0 & 0 & 0\\[-5pt]
    	0 & 1 & -1 & 0 & -1 & -1 & 0 & 0 & 0 & 0 & 0\\[-5pt]
    	0 & 0 & 1 & 0 & 0 & 0 & 0 & -1 & 0 & 0 & 0\\[-5pt]
    	0 & 0 & 0 & -1 & 0 & 0 & 0 & 0 & 0 & 0 & 0\\[-5pt]
    	0 & 0 & 0 & 1 & 0 & 0 & 0 & 0 & 0 & 0 & 0\\[-5pt]
    	0 & 0 & 0 & 0 & 1 & 0 & 0 & 0 & 0 & 0 & 0\\[-5pt]
    	0 & 0 & 0 & 0 & 0 & 1 & -1 & 0 & 0 & 0 & 0\\[-5pt]
    	0 & 0 & 0 & 0 & 0 & 0 & 0 & 1 & -1 & -1 & 0\\[-5pt]
    	0 & 0 & 0 & 0 & 0 & 0 & 0 & 0 & 1 & 0 & 0\\[-5pt]
    	0 & 0 & 0 & 0 & 0 & 0 & 0 & 0 & 0 & 1 & -1\\[-5pt]
    	0 & 0 & 0 & 0 & 0 & 0 & 0 & 0 & 0 & 0 & 1\\[-5pt]
    };
    \draw[dashed] (-1.4,-2.15) to (-1.4,2.3);
    \draw[dashed] (0.65,-2.15) to (0.65,2.3);
    \draw[<->,blue] (-0.6,2.35) to (-0.6,2.45) to [out=90, in=90] (-1.1,2.45) to [->] (-1.1,2.35);
    \begin{pgfonlayer}{background}
    \fill [Gray] (-0.85,2.3) rectangle (-0.35,-2.21);
    \end{pgfonlayer}
    \end{tikzpicture}
    \end{eqnarray*}
  \caption{Moving the critical column (7\textsuperscript{th} column in this case) to the left by series of column interchanges}
  \label{fig:4}
\end{figure}

At this juncture, not only all but one column are upper-triangularized, the diagonal entries of the upper-triangular columns are unity. If we now can upper-triangularize the critical column, its diagonal entry will determine the determinant. Thus the determinant would depend on a single entry after the matrix is made upper-triangular. As it turns out, by adding some of the middle columns to the critical column this can easily be achieved and this is where we make use of Theorem~\ref{theorem:9} which specifically says which columns should be added to the critical column and what the result would be. Let us now start the mathematics to prove the result for a general case.

\subsection{Upper-triangularization of middle columns}\label{subsec:upper_triangularization_middle_col}
\begin{definition}\label{definition:6_1}
	Consider moving the $(j-1)$\textsuperscript{th} column to the $i$\textsuperscript{th} position (moving the $7$\textsuperscript{th} column to the $4$\textsuperscript{th} position in the current example) by a series of column interchange as shown in Fig.~\ref{fig:4}. The resultant matrix is denoted as $\widetilde{\mathbf{M}}_{j,i}$.
\end{definition}

\begin{lemma}\label{lemma:10}
	The columns of $\widetilde{\mathbf{M}}_{j,i}$ are given by $C_k^{\widetilde{\mathbf{M}}_{j,i}}=C_{k}^{\mathbf{M}_{j,i}} ~\forall~k<i$,	 $C_{j-1}^{\mathbf{M}_{j,i}}~\textrm{for}~k=i$, $C_{k-1}^{\mathbf{M}_{j,i}} ~\forall~i<k<j$ and $C_{k}^{\mathbf{M}_{j,i}} ~\forall~k\geq j$.
	\begin{proof}
		The result is obvious for $k<i$ and $k\geq j$ since these columns occupy the same position in $\mathbf{M}_{j,i}$ and $\widetilde{\mathbf{M}}_{j,i}$. By definition, the $(j-1)$\textsuperscript{th} column of $\mathbf{M}_{j,i}$ is shifted to $i$\textsuperscript{th} column in $\widetilde{\mathbf{M}}_{j,i}$ and therefore $C_k^{\widetilde{\mathbf{M}}_{j,i}}=C_{j-1}^{\mathbf{M}_{j,i}}~\textrm{for}~k=i$. By this process, the columns of $\mathbf{M}_{j,i} ~\forall~i\leq k<j$ get respectively shifted to their immediate right position and therefore $C_k^{\widetilde{\mathbf{M}}_{j,i}}=C_{k-1}^{\mathbf{M}_{j,i}} ~\forall~i<k<j$.
	\end{proof}
\end{lemma}

\begin{corollary}\label{corollary:10_1}
	All but the $i$\textsuperscript{th} column of $\widetilde{\mathbf{M}}_{j,i}$ are upper-triangular.
	\begin{proof}
		Here it is sufficient to prove that $C_k^{\widetilde{\mathbf{M}}_{j,i}}~\forall~i<k<j$ are upper-triangular. From Lemma~\ref{lemma:10}, $C_k^{\widetilde{\mathbf{M}}_{j,i}}=C_{k-1}^{\mathbf{M}_{j,i}} ~\forall~i<k<j$.		According to Lemma~\ref{lemma:6},
		\begin{equation*}
		C_k^{\widetilde{\mathbf{M}}_{j,i}}(x)=C_{k-1}^{\mathbf{M}_{j,i}}(x)=
		\left\{
		\begin{array}{ll}
		1~\textrm{if}~x=k\\
		0~\textrm{if}~x>k
		\end{array}
		\right.
		\end{equation*}
		Thus,  $C_k^{\widetilde{\mathbf{M}}_{j,i}}$ are upper-triangular for $i<k<j$ with diagonal entries unity.
	\end{proof}
\end{corollary}

\begin{lemma}\label{lemma:11}
	$|\widetilde{\mathbf{M}}_{j,i}|= (-1)^{j-i-1}|\mathbf{M}_{j,i}|$.
	\begin{proof}
		The column $C_{j-1}^{\mathbf{M}_{j,i}}$ is shifted to occupy the $i$\textsuperscript{th} column position. This series of column interchanges involves total $j-i-1$ column-interchange operations. Consequently, the determinant of the resulting matrix $\widetilde{\mathbf{M}}_{j,i}$ will be given by $|\widetilde{\mathbf{M}}_{j,i}|= (-1)^{j-i-1}|\mathbf{M}_{j,i}|$.
	\end{proof}
\end{lemma}

\subsection{Upper-triangularization of the $i$\textsuperscript{th} column}\label{subsec:upper_triangularization_critical_col}
The $i$\textsuperscript{th} column can be made upper-triangular by suitably adding other columns of $\widetilde{\mathbf{M}}_{j,i}$ to it. How and, more importantly, which columns should be added to it can be found out using Theorem~\ref{theorem:4}. It is also possible to determine what will the diagonal entry of this $i$\textsuperscript{th} column be after it has been made upper-triangular. To demonstrate this fact, consider the following two cases:
\begin{itemize}
	\item \emph{$n_i$ is forward connected to $n_j$}: By Theorem~\ref{theorem:4}, there exist columns $C^{\mathbf{A}}_{k_1}+C^{\mathbf{A}}_{k_2}+\cdots+C^{\mathbf{A}}_{k_p}+C^{\mathbf{A}}_j=X_j^N-X_i^N$ where $i<k_1<k_2<...<k_p<j$. Now, by Theorem~\ref{theorem:9}, $C_{k_1-1}^{\mathbf{M}_{j,i}}+C_{k_2-1}^{\mathbf{M}_{j,i}}+...+C_{k_p-1}^{\mathbf{M}_{j,i}}+C_{j-1}^{\mathbf{M}_{j,i}}=-X_{i}^{N-1}$. Observe that,  $C_{j-1}^{\mathbf{M}_{j,i}}=C_i^{\widetilde{\mathbf{M}}_{j,i}}$ according to Lemma~\ref{lemma:10}. Also, by using  Lemma~\ref{lemma:10}, $C_{k_1-1}^{\mathbf{M}_{j,i}}=C_{k_1}^{\widetilde{\mathbf{M}}_{j,i}}$, $C_{k_2-1}^{\mathbf{M}_{j,i}}=C_{k_2}^{\widetilde{\mathbf{M}}_{j,i}}$, ..., $C_{k_p-1}^{\mathbf{M}_{j,i}}=C_{k_p}^{\widetilde{\mathbf{M}}_{j,i}}$. Therefore,	 $$C_{k_1}^{\widetilde{\mathbf{M}}_{j,i}}+C_{k_2}^{\widetilde{\mathbf{M}}_{j,i}}+...+C_{k_p}^{\widetilde{\mathbf{M}}_{j,i}}+C_i^{\widetilde{\mathbf{M}}_{j,i}}=-X_{i}^{N-1}$$
	Thus by adding $C_{k_1}^{\widetilde{\mathbf{M}}_{j,i}}$, $C_{k_2}^{\widetilde{\mathbf{M}}_{j,i}}$, ..., $C_{k_p}^{\widetilde{\mathbf{M}}_{j,i}}$ to it, $C_i^{\widetilde{\mathbf{M}}_{j,i}}$ can be made upper-triangular and in this process its diagonal entry will be "-1".
	\item \emph{$n_i$ is not forward connected to $n_j$}: By Corollary~\ref{corollary:4_1}, there exist columns $C^{\mathbf{A}}_{k_1}+C^{\mathbf{A}}_{k_2}+\cdots+C^{\mathbf{A}}_{k_p}+C^{\mathbf{A}}_j=X_j^N-X_{i'}^N$ where $i'<i<k_1<k_2<...<k_p<j$. Now, following the same procedures as mentioned above it can be shown that	 $$C_{k_1}^{\widetilde{\mathbf{M}}_{j,i}}+C_{k_2}^{\widetilde{\mathbf{M}}_{j,i}}+...+C_{k_p}^{\widetilde{\mathbf{M}}_{j,i}}+C_i^{\widetilde{\mathbf{M}}_{j,i}}=-X_{i'}^{N-1}$$
	Thus by adding $C_{k_1}^{\widetilde{\mathbf{M}}_{j,i}}$, $C_{k_2}^{\widetilde{\mathbf{M}}_{j,i}}$, ..., $C_{k_p}^{\widetilde{\mathbf{M}}_{j,i}}$ to it, $C_i^{\widetilde{\mathbf{M}}_{j,i}}$ can be made upper-triangular and in this process its diagonal entry will be "0".
\end{itemize}

\subsection{Determinant of $\mathbf{M}_{j,i}$}\label{subsec:det_Mji}
\begin{definition}\label{definition:6_2}
	Let the matrix resulting from diagonalization of the $i$\textsuperscript{th} column of $\widetilde{\mathbf{M}}_{j,i}$ be denoted as $\overline{\mathbf{M}}_{j,i}$.
\end{definition}
Since the addition of columns do not change the determinant, $|\widetilde{\mathbf{M}}_{j,i}|=|\overline{\mathbf{M}}_{j,i}|$. Therefore, according to Lemma~\ref{lemma:11},
\begin{equation}\label{eq_last_det}
|\overline{\mathbf{M}}_{j,i}|=|\widetilde{\mathbf{M}}_{j,i}|=(-1)^{j-i-1}|\mathbf{M}_{j,i}|\Rightarrow |\mathbf{M}_{j,i}|=(-1)^{i+1-j}|\overline{\mathbf{M}}_{j,i}|
\end{equation}

\begin{lemma}\label{lemma:12}
	 $|\mathbf{M}_{j,i}|=(-1)^{i-j}$ if $n_i$ is forward connected to $n_j$, else $|\mathbf{M}_{j,i}|=0$.
	 \begin{proof}
	 	All but the $i$\textsuperscript{th} column of $\widetilde{\mathbf{M}}_{j,i}$ are upper-triangular and the diagonal entries are unity. The same is true for $\overline{\mathbf{M}}_{j,i}$. Additionally for $\overline{\mathbf{M}}_{j,i}$, the diagonal entry is "-1" if $n_i$ is forward connected to $n_j$ and "0" if $n_i$ is not forward connected to $n_j$. Consequently, $|\overline{\mathbf{M}}_{j,i}|=-1$ when $n_i$ is forward connected to $n_j$ and $|\overline{\mathbf{M}}_{j,i}|=0$ otherwise. Therefore, using (\ref{eq_last_det}), $|\mathbf{M}_{j,i}|=(-1)^{i+1-j}\times(-1)=(-1)^{i-j}$ if $n_i$ is forward connected to $n_j$, else $|\mathbf{M}_{j,i}|=0$.
	\end{proof}
\end{lemma}

\begin{theorem}\label{theorem:13}
	 $\mathbf{A}^{-1}(i,j)=1$ if and only if $n_i\in R_{0\rightarrow j}$.
	\begin{proof}
	According to (\ref{eq:A_minor}), $\mathbf{A}^{-1}(i,j)=(-1)^{i+j}|\mathbf{M}_{j,i}|$. Now, using Lemma~\ref{lemma:12}, $\mathbf{A}^{-1}(i,j)=1$ if $n_i$ is forward connected to $n_j$, i.e., $n_i\in R_{0\rightarrow j}$ and  $\mathbf{A}^{-1}(i,j)=1$ if $n_i\notin R_{0\rightarrow j}$.
	\end{proof}		
\end{theorem}

\section{Discussion and Summary}
Computed $\mathbf{A}^{-1}$ for the considered case is shown in \eqref{eq:Ainverse} accompanied by the graph beside it. The proved result can be easily seen to be true for the considered case.

\vspace{0.1in}

\begin{minipage}[r]{0.3\textwidth}
    \begin{tikzpicture}
    \draw (0,0)node[ocirc]{} to (0,-1)node[ocirc]{} to (0,-2) to (0,-3) to (0,-4)node[ocirc]{} to (0,-5)node[ocirc]{};
    \draw (0,-2) to (1,-2)node[ocirc]{};
    \draw (0,-2)node[ocirc]{} to (-1,-2)node[ocirc]{} to (-2,-2)node[ocirc]{};
    \draw (0,-3)node[ocirc]{} to (-1,-3) to (-2,-3)node[ocirc]{};
    \draw (-1,-3)node[ocirc]{} to (-1,-4)node[ocirc]{} to (-1,-5)node[ocirc]{};
    \node [right=1pt] at (0,0) {$n_0$};
    \node [right=1pt] at (0,-1) {$n_1$};
    \node [right=1pt] at (0,-2.2) {$n_2$};
    \node [right=1pt] at (0,-3) {$n_3$};
    \node [right=1pt] at (0,-4) {$n_4$};
    \node [right=1pt] at (0,-5) {$n_5$};
    \node [right=1pt] at (1,-2) {$n_6$};
    \node [above=1pt] at (-1,-2) {$n_7$};
    \node [left=1pt] at (-2,-2) {$n_8$};
    \node [above=1pt] at (-1,-3) {$n_9$};
    \node [left=1pt] at (-2,-3) {$n_{10}$};
    \node [left=1pt] at (-1,-4) {$n_{11}$};
    \node [left=1pt] at (-1,-5) {$n_{12}$};
    \filldraw[draw=black,fill=black] (0,-0.6) --+ (0.1,0.2) --+ (-0.1,0.2);
    \filldraw[draw=black,fill=black] (0,-1.6) --+ (0.1,0.2) --+ (-0.1,0.2);
    \filldraw[draw=black,fill=black] (0,-2.6) --+ (0.1,0.2) --+ (-0.1,0.2);
    \filldraw[draw=black,fill=black] (0,-3.6) --+ (0.1,0.2) --+ (-0.1,0.2);
    \filldraw[draw=black,fill=black] (0,-4.6) --+ (0.1,0.2) --+ (-0.1,0.2);
    \filldraw[draw=black,fill=black] (0.75,-2) --+ (-0.2,0.1) --+ (-0.2,-0.1);
    \filldraw[draw=black,fill=black] (-0.6,-2) --+ (0.2,0.1) --+ (0.2,-0.1);
    \filldraw[draw=black,fill=black] (-1.6,-2) --+ (0.2,0.1) --+ (0.2,-0.1);
    \filldraw[draw=black,fill=black] (-0.6,-3) --+ (0.2,0.1) --+ (0.2,-0.1);
    \filldraw[draw=black,fill=black] (-1.6,-3) --+ (0.2,0.1) --+ (0.2,-0.1);
    \filldraw[draw=black,fill=black] (-1,-3.6) --+ (0.1,0.2) --+ (-0.1,0.2);
    \filldraw[draw=black,fill=black] (-1,-4.6) --+ (0.1,0.2) --+ (-0.1,0.2);
    \end{tikzpicture}
\end{minipage}
\begin{minipage}[c]{0.65\textwidth}
     \begin{equation}\label{eq:Ainverse}
        \mathbf{A}^{-1}=
        \begin{tikzpicture}[baseline=-0.8ex, font=\fontsize{0.105in}{0.1in}\selectfont, every left delimiter/.style={xshift=0.5em},every right delimiter/.style={xshift=-0.5em}]
        \matrix (m) [matrix of math nodes,left delimiter={[},right delimiter={]},row sep=0.7em,column sep=0.3em,text height=0.5ex, text depth=0.2ex]{
            1 & 1 & 1 & 1 & 1 & 1 & 1 & 1 & 1 & 1 & 1 & 1\\[-5pt]
        	0 & 1 & 1 & 1 & 1 & 1 & 1 & 1 & 1 & 1 & 1 & 1\\[-5pt]
        	0 & 0 & 1 & 1 & 1 & 0 & 0 & 0 & 1 & 1 & 1 & 1\\[-5pt]
        	0 & 0 & 0 & 1 & 1 & 0 & 0 & 0 & 0 & 0 & 0 & 0\\[-5pt]
        	0 & 0 & 0 & 0 & 1 & 0 & 0 & 0 & 0 & 0 & 0 & 0\\[-5pt]
        	0 & 0 & 0 & 0 & 0 & 1 & 0 & 0 & 0 & 0 & 0 & 0\\[-5pt]
        	0 & 0 & 0 & 0 & 0 & 0 & 1 & 1 & 0 & 0 & 0 & 0\\[-5pt]
        	0 & 0 & 0 & 0 & 0 & 0 & 0 & 1 & 0 & 0 & 0 & 0\\[-5pt]
        	0 & 0 & 0 & 0 & 0 & 0 & 0 & 0 & 1 & 1 & 1 & 1\\[-5pt]
        	0 & 0 & 0 & 0 & 0 & 0 & 0 & 0 & 0 & 1 & 0 & 0\\[-5pt]
        	0 & 0 & 0 & 0 & 0 & 0 & 0 & 0 & 0 & 0 & 1 & 1\\[-5pt]
        	0 & 0 & 0 & 0 & 0 & 0 & 0 & 0 & 0 & 0 & 0 & 1\\[-5pt]};
        \end{tikzpicture}
     \end{equation}
\end{minipage}

\vspace{0.1in}

In final summary, the properties of forward adjacency matrix $\mathbf{A}$ (and its inverse) that have been proved are as follows -
\begin{itemize}
  \item Entries of $\mathbf{A}^{-1}$ are either 0 or 1.
  \item $\mathbf{A}$ is upper-triangular and consequently $\mathbf{A}^{-1}$ too.
  \item If we consider a particular node $n_k$, $k$\textsuperscript{th} column of $\mathbf{A}$ denotes which node is forward adjacent to it and $k$\textsuperscript{th} column of $\mathbf{A}^{-1}$ denotes all the nodes that are forward connected to it. In other words, $k$\textsuperscript{th} column of $\mathbf{A}^{-1}$ says $n_k$ can be arrived at starting from the \emph{origin}, i.e. $n_0$.
\end{itemize}

Before finishing, the authors would like to draw readers' attention to view the matrices $\mathbf{A}$ and $\mathbf{A}^{-1}$ row-wise. So far, starting from the definition of $\mathbf{A}$, we have viewed the matrices columnwise, however, the rows can be used as well to gain some insight into the graph. The $k$\textsuperscript{th} row of $\mathbf{A}$ denotes which nodes $n_k$ is forward adjacent to. Note that the graph is allowed to arbitrarily branch out from a particular node. This means a node can be forward adjacent to multiple nodes (but multiple nodes cannot be forward adjacent to the same node) and those are denoted with "-1" in a particular row. For instance, $n_2$ is forward adjacent to $n_3$, $n_6$, and $n_7$ in the considered case. Correspondingly, the second row has "-1" exactly as its 3\textsuperscript{rd}, 6\textsuperscript{th} and 7\textsuperscript{th} entry.

Now, let us focus on the rows of $\mathbf{A}^{-1}$. The $k$\textsuperscript{th} row lists all the nodes that $n_k$ is forward connected to. In other words, the $k$\textsuperscript{th} row says which nodes we can approach by travelling forward when we are standing at $n_k$. For instance, $n_5$, $n_6$, $n_8$, $n_{10}$ and $n_{12}$ are the terminal nodes of radial branches and therefore they are not forward connected to any node at all. Correspondingly, we see that the respective rows of $\mathbf{A}$ has only the diagonal as '1' and rest are zeros. Similarly, Row 4, says, node 4 is forward connected to node 5 ONLY, with a '1' at appropriate column index. And row 1 indicates that it is forward connected to all the rest of the nodes, and so has a '1' at all positions.

In this way, the forward adjacency matrix and its inverse provide many visualizations of a radial graph. Recently, the authors found the column property of $\mathbf{A}^{-1}$ sacrosanct to prove some useful results in circuit theory. Branin reported that it was found particularly useful for programming transient analysis of RLC networks and other computer programs as well \cite{branin}.

\bibliographystyle{unsrt}

\end{document}